\documentclass[10pt,reqno]{amsart}
\usepackage{graphicx, amssymb, color,pdfsync}
\usepackage[all,cmtip]{xy}
\numberwithin{equation}{section}
\usepackage{mathrsfs}
\usepackage{tikz}
\usetikzlibrary{matrix,arrows}
\usepackage{dsfont}
\usepackage[USenglish]{babel}

\begin{document}
\newcommand{\s}{\vspace{0.2cm}}

\newtheorem{theo}{Theorem}
\newtheorem{prop}[theo]{Proposition}
\newtheorem{coro}[theo]{Corollary}
\newtheorem{lemm}[theo]{Lemma}
\newtheorem{example}[theo]{Example}
\theoremstyle{remark}
\newtheorem{rema}[theo]{\bf Remark}
\newtheorem{defi}[theo]{\bf Definition}
\newcommand{\rac}{{\mathbb{Q}}}
\newcommand{\comp}{{\mathbb{C}}}
\newcommand{\hip}{{\mathbb{H}}}
\newcommand{\hola}{{\overline{\rac}}}

\title[Families, Uniformization and Arithmeticity]{Families of Riemann Surfaces,\\ Uniformization and Arithmeticity}
\author{Gabino Gonz\'alez-Diez}
\address{Departamento de Matem\'aticas, Universidad Aut\'onoma de Madrid, Madrid, Spain.}
\email{gabino.gonzalez@uam.es}
\author{Sebasti\'an Reyes-Carocca}
\address{Departamento de Matem\'aticas y Estad\'istica, Universidad de La Frontera, Temuco, Chile.}
\email{sebastian.reyes@ufrontera.cl}

\thanks{Both authors were partially supported by Spanish MEyC Grant MTM 2012-31973. The second author was also partially supported   by Becas Chile, Universidad de La Frontera, Fondecyt
Postdoctoral Project \#3160002 and Fondecyt Anilo Project ACT 1415 PIA Conicyt.}


\begin{abstract}
A consequence of the results of Bers and Griffiths on the uniformization of complex algebraic varieties  is that
the universal cover  of
a family of Riemann surfaces,   with base and fibers of finite hyperbolic type,
is  a  contractible $2-$dimensional domain that can be realized as   the graph of  a holomorphic motion of the unit disk.

In this paper we   determine which holomorphic motions
give rise to these uniformizing domains
and characterize  which among them correspond to arithmetic families (i.e. families defined over number fields).
Then we apply these results to
characterize the arithmeticity of
complex surfaces of general type
in terms of the biholomorphism class of the  $2-$dimensional domains that arise as universal covers of their
Zariski open subsets. For the important class of Kodaira fibrations this criterion
implies that arithmeticity  can be read off from the universal cover.
 All this is very much in contrast with the corresponding situation in complex dimension one, where the universal cover is always the unit disk.
\end{abstract}
\maketitle

\section{Introduction and Statement of the main Results} Let $V \to C$ be a holomorphic family of Riemann surfaces over a Riemann surface $C$ such that the base  and the fibers are of finite hyperbolic type. Bers-Griffiths' Uniformization Theorem implies that the holomorphic universal cover (from now on, simply   the universal cover) of $V$ is a
 contractible bounded domain   $\mathscr{B} \subset \comp^2$ isomorphic to the graph of (the extension to the unit disk of) a  holomorphic motion of the unit circle $\Delta \times \mathbb{S}^1 \to  \comp.$ (Domains of this kind are termed Bergman domains in Bers' article   \cite[p. 284]{Bers}).

In this paper we characterize the   domains
 $\mathscr{B}$
that arise as universal covers of families of Riemann surfaces as above (Theorem \ref{theo0}). Moreover, we provide a criterion (Theorem \ref{coro})
   to recognize which among them   correspond to
  arithmetic families, that is, families defined over a number field (see Subsection \ref{ari} for the precise definition).

While it is classically known that the universal cover of a Riemann surface is isomorphic to the Riemann sphere $\overline{\comp},$ the complex plane $\comp$ or the unit disk $\Delta,$ there seems to be  a huge amount of possibilities for the  universal cover of a complex surface. In the final part of this paper we apply the aforementioned results to show that
 the arithmeticity (that is, the property of being definable over a number field) of a minimal complex  projective surface of general type $S$ depends only on whether or not it contains a Zariski open subset $U\subset S$ whose universal cover is isomorphic to  a contractible bounded domain  of a certain biholomorphic type
  (Theorem \ref{aritmetica}).

When $S$ is a Kodaira fibration, in Theorem \ref{aritmetica} one can specify $U= S$. This fact allows us   to   characterize the arithmeticity of  Kodaira fibrations
in terms of the isomorphism class of their universal covers
  (Theorem \ref{aritmeticaK}). In turn, Theorem \ref{aritmeticaK}  strengthens   the main result in \cite{Koda1} which says that two Kodaira fibrations with isomorphic universal covers are simultaneously
arithmetic or non-arithmetic.


In conclusion we remark that in this respect projective surfaces differ radically from projective curves for which the biholomorphism class of the  universal cover is a topological invariant.

\subsection{The group sequence associated to a holomorphic motion   of $\mathbb{S}^1$.} \label{BBA}

Let $E \subset \overline{\comp}$
be a subset of cardinality $\geq 3$. A {\it holomorphic motion} of $E$  is a function $W : \Delta \times E \to \overline{\comp}$ satisfying the following properties:
\begin{enumerate}
\item[(a)] $W(0, s)=s$ for all $s \in E.$
\item[(b)] $t \mapsto W(t,s)$ is holomorphic for all $s \in E.$
\item[(c)] $s \mapsto W_t(s):=W(t,s)$ is injective for all $t \in \Delta.$
\end{enumerate}

Holomorphic motions were introduced by Ma\~{n}\'e, Sad and Sullivan in \cite{MSS}. They proved the first important result in the topic, the so-called {\it $\lambda-$lemma}, which says that $W$ is actually a continuous map and that, moreover, the functions $W_t$ are quasiconformal. Probably the deepest  result in the subject is Slodkowski's Theorem \cite{SL} which establishes that each holomorphic motion of $E$   admits an extension to a holomorphic motion of the whole Riemann sphere. We refer to the survey articles \cite{Gardiner}  and \cite{mitra} for more information on this topic.

We will be concerned with holomorphic motions of the  unit circle
 $W: \Delta \times \mathbb{S}^1 \to \comp$ and their extensions to
 the unit disk $W: 	\Delta \times \Delta \to \comp$.
  An obvious consequence of the  $\lambda-$lemma and Slodkowski's Theorem is that a holomorphic motion $W$ of $\Delta$ determines a unique holomorphic motion of $\mathbb{S}^1 $ which we will still denote by $W$.

The {\it graph of $W$}  is the domain
$$ \mathscr{B}_W
	:=\{(t,W(t,z)) : t, z \in \Delta \}.$$
	
	We observe that each of the {\it quasidisks}
	$$
	D_t:=\{W(t,z): z \in \Delta \}
	$$
	 is independent of
   the choice of the extension of $W$ to $\Delta$
	and therefore so must be
	 the domain $ \mathscr{B}_W$. Now the property $(b)$ and the $\lambda-$lemma together imply that the projection to the first coordinate $\mathscr{B}_W \to \Delta$
	makes of $\mathscr{B}_W$   a  family of quasidisks (see  the precise definition   in the Subsection \ref{DefFamily} below)
	  whose fiber over each $t \in \Delta$ is precisely $D_t.$

	If   $\pi:\mathscr{B} \to \Delta$ (also written
	$(\mathscr{B}, \pi)$) is a family of quasidisks
	we denote by $\mbox{Aut}(\mathscr{B})$ the group of biholomorphic automorphisms of $\mathscr{B}$ and by $\mbox{Aut}_{\pi}(\mathscr{B})$ the subgroup consisting of the fiber preserving ones (i.e. those that map each $D_t$ into another fiber $D_{t'}$). Each $\varphi \in \mbox{Aut}_{\pi}(\mathscr{B})$ induces an automorphism $\hat{\varphi}$ of the unit disk given by the rule $D_t \mapsto D_{\hat{\varphi}(t)}:=\varphi(D_t).$ The correspondence $\varphi \mapsto \hat{\varphi}$ defines a homomorphism $\Theta :\mbox{Aut}_{\pi}(\mathscr{B}) \to \mbox{Aut}(\Delta)$ which induces an obvious exact sequence of groups
	
	\[ \xymatrix {
  1 \ar[r] &
  \mathbb{K}_{\pi} \ar[r] &
  \mbox{Aut}_{\pi}(\mathscr{B}) \ar[r]^{\,\,\,\,\,\,\,\,\,\,\,\Theta} &
  \Gamma_{\pi} \ar[r] &
  1
} \]

The restriction of  $\mathbb{K}_{\pi}$ to each fiber $D_t$ induces a group homomorphism
	\begin{equation} \label{restriction}
	\Phi_t : \mathbb{K}_{\pi} \to \mbox{Aut}(D_t)
		\end{equation}
which by  a result of Earle and Marden
  \cite[Corollary 5.1]{EA} is injective.
    We shall denote by $K_{\pi}^t$ the image of $\Phi_t.$ For the particular case $t=0,$ the monomorphism
     $$\Phi_0 : \mathbb{K}_{\pi} \to \mbox{Aut}(D_0)$$
     maps $\mathbb{K}_{\pi}$ isomorphically onto a subgroup $K_{\pi}^0$ of $\mbox{Aut}(\Delta)$ which we shall denote simply  by $K_{\pi}.$ Obviously this fact allows us to rewrite the above group sequence in the following    equivalent form
	\[ \xymatrix {
  1 \ar[r] &
  K_{\pi} \ar[r] &
  \mbox{Aut}_{\pi}(\mathscr{B}) \ar[r]^{ \,\, \,\,\Theta} &
  \Gamma_{\pi} \ar[r] &
  1.
} \]	
 In the case in which  $(\mathscr{B}, \pi)$ is the family induced by a holomorphic motion $W$ we put  $K_{W}^t$ instead of $K_{\pi}^t$ and write the previous sequence as
\[ \xymatrix {
  1 \ar[r] &
  K_{W} \ar[r] &
  \mbox{Aut}_{\pi}(\mathscr{B}_{W}) \ar[r]^{\,\,\,\,\,\,\,\,\,\,\,\Theta} &
  \Gamma_{W} \ar[r] &
  1.
} \] We refer to  this sequence as the  {\it sequence associated} to the holomorphic motion $W$
and to the
 subgroups   $K_{W}$ and   $\Gamma_{W} $ of $\mbox{Aut}(\Delta)$ as
the {\it base} and the {\it fiber groups associated} to $W.$

We will say that     a
	holomorphic motion $W$  {\it is trivial} if
	 $\mathscr{B}_{W}$ is isomorphic to the bidisk $\Delta \times \Delta$.

In order to state our results we will also need to bring in the concept of equivariancy. Let $K$ be a group of M\"{o}bius transformations that leaves $\mathbb{S}^1$ invariant. A holomorphic motion $W: \Delta \times \mathbb{S}^1 \to \overline{\comp}$ is called $K-${\it equivariant} if for each $t \in \Delta$ there is a group homomorphism
$$
X_t:K \to \mathbb{P}\mbox{SL}(2, \mathbb{C})
$$
such that
for every $\kappa \in K$ and every  the  $s \in \mathbb{S}^1$  the following identity is satisfied
$$W(t, \kappa(s))=X_t(\kappa)(W(t,s)).$$

For later use we record here the fact, proved by
    Earle, Kra and Krushkal (see \cite[Theorem 1]{Earle}),    that if a holomorphic motion is $K-$equivariant then there exists such an extension to $\overline{\comp}$ which is $K-$equivariant as well.

\subsection{Holomorphic motions and families of Riemann surfaces.}\label{DefFamily}
 Let $V$ be a complex surface (that is, a two-dimensional complex analytic manifold) and $C$ be a Riemann surface. Following Hubbard (\cite[Section 6.2]{Hubbard}; see also \cite{EA}) a
holomorphic map $f: V \to C$ will be  called a {\it holomorphic family of Riemann surfaces}   if
the following conditions are satisfied:
\begin{enumerate}
\item[(a)] $f$ is everywhere of maximal rank, so that the fibers
$ V_c:=f^{-1}(c), c\in C,$ are Riemann surfaces.
\item[(b)] $f$ locally admits {\it horizontally  holomorphic trivializations}.
\end{enumerate}

Condition $(b)$ means that for every $c_0\in C$ there is a neighborhood  $B$ of $c_0$ and a homeomorphism
 $\theta:B\times V_{c_0} \to f^{-1}(B) $ commuting with the projection to $B,$ such that for every $x\in  V_{c_0}$ the map $c\to \theta(c,x)$ is holomorphic. This condition may look unnatural at first sight, but note that it rules out unwanted families such as the one whose fiber over each $z\in \Delta$ is $ \comp \setminus \{0,1,-1,-1+\overline{z}  \}$
 (\cite[Example 6.2.10]{Hubbard}).

It should be observed that, if the fibers $V_c$ are compact,
   condition (b) is automatically satisfied and that, if the fibers are Riemann surfaces of finite type $(g,n)$, i.e. with genus $g$ and $n$ punctures, then
  $f: V \to C$ is the restriction of a family $ V^+\to C$ with compact fibers to the complement of the images of
  $n$  disjoint   holomorphic sections $s_1, \dots, s_n:C \to V^+$
   (\cite[Example 6.2.10]{Hubbard}; c.f. \cite[p. 346]{Nag}).

  We  will reserve the expression {\it algebraic  family of Riemann surfaces}
  for families of Riemann surfaces of finite hyperbolic type (i.e. such that  $2g-2+n>0$) whose base $C$ is also of finite hyperbolic type.
An algebraic family is called {\it isotrivial}
if all its fibers are isomorphic Riemann surfaces.
We anticipate  that algebraic families give rise to algebraic surfaces  (see Section \ref{s12}).

We observe that the graph of a holomorphic motion $W$  of $\mathbb{S}^1$  endowed with the natural projection to the first coordinate  is a family of quasidisks; the required horizontally holomorphic trivializations being provided by  $W$.

We now make the following definitions.
\begin{defi}
A family of quasidisks  $(\mathscr{B}, \pi)$ will be said to be the  {\it the universal cover of an algebraic   family of Riemann surfaces}  $f: V \to C$ if
there is   a commutative diagram
$$
\begin{tikzpicture}[node distance=3.1 cm, auto]
  \node (P) {$\mathscr{B}$};
  \node (Q) [right of=P] {$V$};
  \node (A) [below of=P, node distance=1.2 cm] {$\Delta$};
  \node (C) [below of=Q, node distance=1.2 cm] {$C$};
  \draw[->] (P) to node {} (Q);
  \draw[->] (A) to node {} (C);
  \draw[->] (P) to node [swap] {$\pi$} (A);
    \draw[->] (Q) to node {${f}$} (C);
\end{tikzpicture}
$$
such that the horizontal arrows are universal covering maps.
\end{defi}

\begin{defi}
By a  {\it  Bers-Griffiths domain} \label{definition2}
we will refer to a domain $\mathscr{B} \subset {\comp}^2$
which is the graph of a  non-trivial holomorphic motion $W$  that satisfies  the following
properties:
\begin{enumerate}
\item[(a)] The base and fiber groups $K_{W}$ and $\Gamma_{W}$ are Fuchsian groups of finite hyperbolic type.
\item[(b)]  $W$  is $K$-equivariant, for some finite index subgroup $K$ of $K_{W}.$
\end{enumerate}

 Accordingly, by a {\it   Bers-Griffiths family} (of quasidisks) we will refer to   a Bers-Griffiths domain  endowed with the projection to the first coordinate.
\end{defi}

\begin{theo} \label{theo0}
A family of quasidisks  $\pi:\mathscr{B} \to \Delta$ is the universal cover of a non-trivial algebraic family of Riemann surfaces if and only if it is isomorphic to
a Bers-Griffiths family.
\end{theo}

This theorem will be proved in Section \ref{MH}. In that section we will also prove  that
the holomorphic motion of which a Bers-Griffiths domain $\mathscr{B}$ is the graph is uniquely determined by $\mathscr{B}$  (Proposition \ref{uniquely}).
 This uniqueness property, whose proof will be an easy
 consequence of work   of Earle and Fowler \cite{ERF},  will allow us to easily produce examples of holomorphic motions whose
  graphs    are not   Bers-Griffiths domains (Example  \ref{counterexample}).

\subsection{Holomorphic motions and arithmetic families of Riemann surfaces} \label{ari}


Let us denote by $\mbox{Gal}(\comp)$  the group of field automorphisms of $\comp.$ Let $X \subset \mathbb{P}^n$ be a projective variety and $\sigma \in \mbox{Gal}(\comp).$ We shall denote by $X^{\sigma}$ the projective variety defined by the polynomials obtained after applying $\sigma$ to the coefficients of the homogeneous   polynomials which define $X.$

Let $k$ be a subfield of $\comp$ and $\mbox{Gal}(\comp/k)$ be the subgroup of $\mbox{Gal}(\comp)$ consisting of those automorphisms which fix the elements in $k.$ We shall say that $X$ is defined over $k$ if $X=X^{\sigma}$ for all $\sigma \in \mbox{Gal}(\comp/k).$ We shall say that $X$ can be defined over $k$ (or that {\it $k$ is a field of definition for $X$}) if there exists an isomorphism $\Phi : X \to Y$ into a projective variety $Y \subset \mathbb{P}^m$ which is defined over $k.$ Let $U$ be
a Zariski open subset of $X.$ We will say that the quasiprojective subvariety $U \subset X$ is defined over $k$ if both $X$ and the Zariski closed set $X \setminus U$ are defined over $k.$ We will say that $U$ can be defined over $k$ if there exists an isomorphism $\Phi : X \to Y$ into a projective variety $Y$ in such a way that $ \Phi(U)$ is a Zariski open subset of $Y$  defined over $k.$

Likewise, we shall say that a regular map $f$ between quasiprojective varieties is defined over $k$ if $f^{\sigma}=f$ for all $\sigma \in \mbox{Gal}(\comp/k)$ and that $f$ can be defined over $k$ if it is equivalent to a regular map defined over $k.$ Here again $f^{\sigma}$ is the regular map obtained after applying $\sigma$ to the polynomials which locally define $f$ (see for example \cite[p. 34]{Shafa}).

In this paper we are primarily interested in quasiprojective varieties and morphisms which can be defined over numbers fields or, equivalently, over the algebraic closure $\overline{\rac}$ of the field of the rational numbers.
 
\begin{defi}
\begin{enumerate}
\item[(a)] A quasiprojective variety $V$ (resp. a morphism between quasiprojective varieties $f$) will be called {\it arithmetic} if $V$ (resp. $f$) can be defined over a number field.
\item[(b)]  Let  $V$ be a quasiprojective surface. 
 An algebraic family of Riemann surfaces $f: V \to C$ will be called {\it arithmetic}  provided that $V, C$ and $f$ are arithmetic.
\end{enumerate}
\end{defi}
 

%

We anticipate that, as a consequence of Shabat's Theorem (see Section \ref{s2}), the quotient space $\Delta/\Gamma_{W}$ corresponding to any
holomorphic motion $W$ whose graph is a Bers-Griffiths domain
has structure of Riemann orbifold of finite type. We will employ the notation $$\mathcal{O}_{W}=(R; q_1, \ldots, q_m)$$
to refer to this structure;
meaning that its underlying Riemann surface is isomorphic to $R$ and that the universal covering map $\Delta \to R$ ramifies over a set of conic points $\{q_1, \ldots, q_m\}.$ We will say that $\mathcal{O}_{W}$ is an {\it arithmetic orbifold} if both $R$ and the set of conic points are arithmetic. We recall that, by definition, $R$ is arithmetic if both its compactification $\overline{R}$ and the
	finite subset
	  $\overline{R} \setminus R$ are.

    \begin{defi} We will say that a  Bers-Griffiths domain  (and its canonically associated    family of quasidisks)   is of {\it arithmetic type}   if,
  in addition to properties (a) and (b) in Definition \ref{definition2},  the defining graph $W$ satisfies  the  following condition:
   \begin{enumerate}
\item[(c)] $\mathcal{O}_{W}$ is an arithmetic orbifold.
\end{enumerate}
    \end{defi}

\begin{theo} \label{coro}
A family of quasidisks $(\mathscr{B}, \pi)$ is the universal cover of an arithmetic family of Riemann surfaces if and only if it is isomorphic  to a Bers-Griffiths family of arithmetic type.
\end{theo}

The typical case in which contition (c) is satisfied occurs when $\Gamma_{W}$ is commensurable to a hyperbolic triangle group (Corollary \ref{coro3}).

\subsection{Arithmetic complex surfaces} Let $S$ be a non-singular minimal projective surface of general type. It is easy to prove (see for example \cite{Gabino}) that if $S$ is arithmetic then there is a rational map $f: S \dashrightarrow \mathbb{P}^1$ with base locus $B \subset S$ and critical values $\mbox{crit}(f) \subset \mathbb{P}^1$ such that the induced map $$f: U:=S \setminus B \setminus f^{-1}(\mbox{crit}(f)) \to \mathbb{P}^1 \setminus \mbox{crit}(f)$$ is an
arithmetic
family of Riemann surfaces. Hence the universal cover of the Zariski open subset $U \subset S$ is a Bers-Griffiths domain of arithmetic type. The following theorem asserts that the converse also holds.

\begin{theo} \label{aritmetica}
Let $S$ be a non-singular minimal projective surface of general type. Then, the following statements are equivalent:
\begin{enumerate}
\item[(a)] $S$ is an arithmetic complex surface.
\item[(b)] Among all  Zariski open subsets of $S$ there is one whose universal cover is isomorphic to
 a Bers-Griffiths domain of arithmetic type.
\end{enumerate}
\end{theo}


We recall that any non-singular projective surface $S$ admits a minimal model $S_{min}$ from which $S$ is obtained by a finite sequence of blow-ups and that
$S$ is arithmetic if and only if $S_{min}$ and the
  finite set of points of $S_{min}$ at which these blow-ups are centered are arithmetic (see \cite{Gabino}).

\

A Kodaira fibration $S \to C$ is a non-isotrivial holomorphic family of compact Riemann surfaces over a compact Riemann surface $C.$  Such complex surfaces were studied by  Kodaira \cite{Kodaira},  Atiyah \cite{Atiyah},  Hirzebruch \cite{Hirzebruch}
  and Kas \cite{Kas} as examples of  differentiable fiber bundles whose signatures  are not multiplicative. It is known that the base   and   the fibers of such  families must be of genus   at least two and three respectively, and that  $S$
 must be a minimal projective surface of general type.

  We will observe that for  Kodaira fibrations, the Zariski open subset in the previous theorem can be taken to be $S$ itself.
This way we obtain the following strengthening of the main theorem
  obtained by the authors in \cite{Koda1}.

\begin{theo} \label{aritmeticaK}
Let $S \to C$ be a Kodaira fibration. Then, the following statements are equivalent:
\begin{enumerate}
\item[(a)] $S$ is an arithmetic complex surface.
\item[(b)] The  universal cover of $S$ is isomorphic to a
Bers-Griffiths domain of arithmetic type.
\end{enumerate}
\end{theo}
In particular, the arithmeticity, and in fact the  (algebraic closure of  the)   field of definition (see the remark below)
of a   Kodaira fibration depends only on the biholomorphic class of its universal cover.

\begin{rema} It is worth observing that, as  in  \cite{Koda1},
Theorems \ref{coro}, \ref{aritmetica} and \ref{aritmeticaK}
can be stated in a
  slightly more general form. As a matter of fact, we can replace the word ``arithmetic'' by ``can be defined over $k$'' where $k$ is any algebraically closed subfield of $\comp.$
\end{rema}

The paper is organized as follows. In Section \ref{s12} we will briefly review the facts of  Teichm\"{u}ller  and Moduli theory that will be needed throughout the paper. In Section \ref{MH} we include results concerning the relationship between holomorphic motions, their graphs and the universal covers of families of Riemann surfaces.
  Sections \ref{s4}, \ref{sec4} and \ref{s7} will be   devoted to prove Theorems  \ref{theo0}, \ref{coro} and \ref{aritmetica} and \ref{aritmeticaK} respectively.

\

This article is based on results contained in the second author's Ph. D. thesis at the Universidad Aut\'onoma de Madrid.

\section{Teichm\"{u}ller and Moduli Theory} \label{s12}

Let $X$ be a Riemann surface  of finite type. A {\it marking} on $X$ is a pair $(h,Y)$ where $Y$ is a Riemann surface and $h: X \to Y$ is a quasiconformal homeomorphism. Two markings $(h_1,Y_1)$ and $(h_2,Y_2)$ are equivalent if there is an isomorphism $\varphi : Y_1 \to Y_2$ of Riemann surfaces such that $h_2^{-1}\varphi h_1$ is homotopically trivial. The {\it Teichm\"{u}ller space} $T(X)$ of $X$ is the set of classes of markings $[(h,Y)]$ on $X$. The point $[(id, X)]$ is usually referred to as the origin of $T(X).$ The {\it mapping class group} $\mbox{Mod}(X)$ is the group of homotopy classes of quasiconformal homeomorphisms of $X.$

Let $G$ be a Fuchsian group. Let us denote by $L_1^{\infty}(\Delta, G)$ the space of measurable functions $\mu : \Delta \to \comp$ with $L^{\infty}$ norm less than one that are $G-$invariant in the sense that for all $\gamma \in G$ the equality $(\mu \circ \gamma) \cdot \overline{\gamma'}/\gamma'=\mu$ holds almost everywhere; its elements are called {\it Beltrami coefficients} of $G.$ By Ahlfors-Bers' Existence Theorem (see for example \cite[p. 34]{Nag}), for each $\mu \in L_1^{\infty}(\Delta, G)$ there exists a unique quasiconformal homeomorphism $w^{\mu}$ of the Riemann sphere fixing $-1, 1, i$ whose complex dilatation $\bar{\partial} w^{\mu}/\partial w^{\mu}$ agrees with $\mu$ in $\Delta,$ vanishes in $|z| > 1$ and is $G-$compatible in the sense that $G^{\mu}:=w^{\mu}G(w^{\mu})^{-1}$ is a group of M\"{o}bius transformations. Two Beltrami coefficients $\mu$ and $\nu$ are  equivalent if the corresponding maps $w^{\mu}$ and $w^{\nu}$ agree on $\partial \Delta=\mathbb{S}^1.$ The {\it Teichm\"{u}ller space} $T(G)$ of $G$ is the set of classes $[\mu]$ of Beltrami coefficients of $G.$ The {\it extended mapping class group} $\mbox{mod}(G)$ of $G$ is the group of quasiconformal homeomorphism of $\Delta$ which normalize $G$ modulo the subgroup consisting of those which extend to the boundary as the identity map. The {\it mapping class group} $\mbox{Mod}(G)$ of $G$ is the quotient group $\mbox{mod}(G)/G.$

\s

Let us now assume that $G$ is a torsion free Fuchsian group such that the quotient $X:=\Delta/G$ is a Riemann surface of finite hyperbolic type $(g,n).$ Let $\mu \in L_1^{\infty}(\Delta, G).$ We shall denote by $X^{\mu}$ the Riemann surface $w^{\mu}(\Delta)/G^{\mu}$ and by $f^{\mu}: X \to X^{\mu}$ the quasiconformal homeomorphism induced by $w^{\mu}.$ It is a classical result that the rule
\begin{equation} \label{bijection}
\mu \mapsto (f^{\mu}, X^{\mu})
\end{equation}
induces a bijection between $T(G)$ and $T(X)$ and also a group isomorphism  between $\mbox{Mod}(G)$ and $\mbox{Mod}(X).$ We notice that the class of the Beltrami coefficient $\mu=0$ in $T(G)$ corresponds to the origin $[(id, X)]$ of $T(X).$ From now on, we shall identify both spaces and both groups and we shall employ the notation $T_{g,n}$ and $\mbox{Mod}_{g,n}$ to refer to the Teichm\"{u}ller space and the mapping class group respectively. Accordingly, we shall use the notation $\mbox{mod}_{g,n}$ to refer to the group $\mbox{mod}(G).$

We recall (see, for example, the third chapter in \cite{Nag}) that,
by Bers' Embedding Theorem, the Teichm\"{u}ller space acquires a structure of finite dimensional complex analytic manifold on which the mapping class group acts properly discontinuously as a group of biholomorphic automorphisms by the formula
$$
f\circ [(h,Y)]= [(h \circ f^{-1},Y)].
$$
 Therefore,  the corresponding quotient space $\mathscr{M}_{g,n}:=T_{g,n}/\mbox{Mod}_{g,n}$ is naturally endowed with a  structure of normal complex analytic space. This space is known as the {\it moduli space} of Riemann surfaces of  finite  type $(g,n)$ and its points correspond bijectively to the set of isomorphism classes of Riemann surfaces (or equivalently, algebraic curves) of finite   type $(g,n).$


Bers also constructed a fiber space  $F_{g,n} \to T_{g,n}$ such that the fiber over a point $[\mu]$ is the quasidisk $w^{\mu}(\Delta).$ More precisely,
\begin{equation} \label{BersFiberspace}
F(G)=F_{g,n}= \{  ([\mu], z)\in T_{g,n}\times \comp :\mu \in L_1^{\infty}(\Delta,G) \text{ and }  z\in w^{\mu}(\Delta) \}.
\end{equation}

 Moreover, the
  action of  $\mbox{Mod}_{g,n}$ on $T_{g,n}$ can be lifted to an action of $\mbox{mod}_{g,n}$ (hence of $G$)  on $ F_{g,n}$ giving
  rise to two important fiber spaces.

  The first one is
  the so-called {\it universal family} (or {\it Teichm\"{u}ller curve})  of finite type $(g,n)$
  $$p_{g,n}: V(G)=V_{g,n}:=F_{g,n}/G \to T_{g,n}.$$
This is a holomorphic fiber space with the property that the fiber over $[\mu]$ is the Riemann surface of finite type $(g,n)$ given by the quotient $w^{\mu}(\Delta)/G^{\mu}.$ We remark that if $[\mu]$ corresponds to $[(h,Y)]$ (via the bijection (\ref{bijection})) then $w^{\mu}(\Delta)/G^{\mu}$ is isomorphic to $Y.$

The second one is a quotient of the first one, namely
$$\pi_{g,n} : \mathscr{C}_{g,n}:=F_{g,n}/\mbox{mod}_{g,n} =V_{g,n}/\mbox{Mod}_{g,n}\to \mathscr{M}_{g,n},$$
 and is usually called
the {\it universal curve} of  type $(g,n)$. This is a
  normal analytic   space with the property that the fiber over a point   representing a Riemann surface $Y$ is a Riemann surface isomorphic to the quotient $Y/\mbox{Aut}(Y).$

\s

Loosely speaking, one can``fill in the $n$ punctures" of the fibers of $p_{g,n}$ and $\pi_{g,n}$ so as to
to obtain  fiber spaces
$$p_{g,n}^+:V_{g,n}^+ \to T_{g,n}$$
and
$$\pi_{g,n}^+:\mathscr{C}_{g,n}^+  \to \mathscr{M}_{g,n}$$
whose fibers are the closures of the corresponding
fibers of
$p_{g,n}$ and
$\pi_{g,n}$. The projections
$p_{g,n}^+$ and $\pi_{g,n}^+$ come equipped with $n$ disjoint holomorphic   global sections given by the position of the $n$ punctures. We shall denote
 those of $\pi_{g,n}^+ $   by
  $$ s_i : {\mathscr{M}}_{g,n} \to {\mathscr{C}}_{g,n}$$
   These fiber spaces can be obtained  
by using a Fuchsian group $G$ with   $n$ conjugacy classes of torsion elements in the construction of the Teichm\"{u}ller space $T(G)$ (see e.g.   \cite[p. 323]{Nag}).

\s

Let now $f: V \to C$ be an algebraic family of Riemann surfaces of finite type $(g,n)$; let us denote by  $p : \Delta \to C$ be the universal covering map of $C.$ Let  $h: p^*V \to \Delta$ be the pull-back family  and   $X$ be the central fiber $h^{-1}(0) \cong f^{-1}(h(0))$, so that we have a commutative diagram as follows.
$$
\begin{tikzpicture}[node distance=3.1 cm, auto]
  \node (P) {$p^*V$};
  \node (Q) [right of=P] {$V$};
  \node (A) [below of=P, node distance=1.2 cm] {$\Delta$};
  \node (C) [below of=Q, node distance=1.2 cm] {$C.$};
  \draw[->] (P) to node {} (Q);
  \draw[->] (A) to node {$p$} (C);
  \draw[->] (P) to node [swap] {$h$} (A);
    \draw[->] (Q) to node {${f}$} (C);
\end{tikzpicture}
$$

It is easy to check that  $h: p^*V \to \Delta$ is again a family of Riemann surfaces  (see e.g. \cite[Section 2.2]{EA}). Since $\Delta$ is contractible, this new family $h$ admites
  a topological trivialization
$H: \Delta \times X \to  p^*V$; we can further assume that
  the restriction of $H$ to $\{0\} \times X$ is the identiy map. If we denote by $H_t$ the restriction of $H$ to $\{t\} \times X$ we can consider the
{\it classifying map}   of $h$
\begin{equation} \label{classifyingmap}
 \tilde{\Phi} : \Delta \to T(X)\equiv T_{g,n}
\end{equation}
given by $t \mapsto [(H_t, h^{-1}(t))]$; note that  $\tilde{\Phi}(0)=[(id, X)].$ By the Universal Property of the Teichm\"{u}ller curve
(see for example \cite[p. 349]{Nag} and \cite[p. 250]{ERF})
$\tilde{\Phi}$ is a holomorphic map such that

\begin{equation} \label{pullback}
  p^*V \cong (\tilde{\Phi})^*V_{g,n}.
\end{equation}

Clearly, the classifying map above induces
a well-defined (also called) classifying map $\Phi = \Phi_f: C \to \mathscr{M}_{g,n}$ of $f$ defined by sending a point $x \in C$ to the point $[ f^{-1}(x) ] \in \mathscr{M}_{g,n}$ representing the isomorphism class of the fiber of $f$ over $x.$

\s

We shall  denote by $\mbox{Mod}_{g,n}^{[d]}$ the {\it level $d$ mapping class group}. This is   the finite index normal subgroup of $\mbox{Mod}_{g,n}$ consisting of those homotopy classes of homeomorphisms of $X$ which induce the identity map on the homology group $H_1(X, \mathbb{Z}/d\mathbb{Z})$. We notice that, by a theorem of Serre (see for example \cite[p. 275]{Farkas}), if $d\ge 3$ then $\mbox{Mod}_{g,n}^{[d]}$ does not contain non-trivial elements of finite order. Thereby, quotienting by this group yields a holomorphic fibration of non-singular complex curves$$\pi_{g,n}^{[d]} : \mathscr{C}_{g,n}^{[d]}:=V_{g,n}/\mbox{Mod}_{g,n}^{[d]} \to \mathscr{M}_{g,n}^{[d]}:=T_{g,n}/\mbox{Mod}_{g,n}^{[d]}$$called the {\it level $d$ universal curve} of finite type $(g,n).$
The advantage of this fibration over  the standard universal curve
$\pi_{g,n} : \mathscr{C}_{g,n}   \to \mathscr{M}_{g,n}$ is that now
  the fiber over a point representing a Riemann surface $Y$ is a Riemann surface isomorphic to $Y$ (instead of a quotient of it).
	In the same way as above, one can construct a fibration
$${\pi}_{g,n}^{+[d]}:\mathscr{C}_{g,n}^{+[d]}
\to \mathscr{M}_{g,n}^{[d]}$$
whose fibers are the closures   of the fibers of ${\pi}_{g,n}^{[d]}:\mathscr{C}_{g,n}^{[d]}
\to \mathscr{M}_{g,n}^{[d]}.$
We shall refer to this fibration as the {\it level $d$ $n-$pointed universal curve} of genus $g.$
As before, this fibation
  possesses $n$ holomorphic  (and, in fact, algebraic; see further down in this section) disjoint global sections
	$$ s_i^{[d]} : {\mathscr{M}}_{g,n}^{[d]} \to {\mathscr{C}}_{g,n}^{+[d]}$$
	one for each of the punctures.
For the sake of explicitness, from now on,  we will let $d$ be equal to three.

\s

Let $\Gamma$ be the  covering group of the universal cover $p: \Delta \to C.$ By the {\it monodromy} of the family $f: V \to C$ we will understand the group  homomorphism $\mathbf{M} : \Gamma \to \mbox{Mod}_{g,n}$  defined by
 $$\tilde{\Phi} \circ \gamma = \mathbf{M}(\gamma) \circ \tilde{\Phi}$$
for $\gamma \in \Gamma.$  Let  $\Gamma_3$ be the preimage of $\mbox{Mod}_{g,n}^{[3]}$ under $\mathbf{M}$ and $C_3$ be the respective quotient Riemann surface $\Delta/\Gamma_3.$ We will denote by $p_3 : C_3 \to C$ the finite degree covering map induced by the inclusion $\Gamma_3 \le \Gamma$ and by $f_3 : V_3 \to C_3$ the pull-back of $f:V\to C$ by $p_3.$ Then, the classifying map $\tilde{\Phi}$ induces a level three classifying map
$$
 \Phi_{f_3} : C_3 \to \mathscr{M}_{g,n}^{[3]}
$$
which permits to recover the family $f_3$ as the pull-back of the level three universal curve of type $(g,n)$ by $\Phi_{f_3}.$ More precisely, there is a commutative diagram, as follows. $$
\begin{tikzpicture}[node distance=3.1 cm, auto]
  \node (P) {$V_3 \cong \Phi_{f_3}^*\mathscr{C}_{g,n}^{[3]}$};
  \node (Q) [right of=P] {${\mathscr{C}}_{g,n}^{[3]}$};
  \node (A) [below of=P, node distance=1.3 cm] {$C_3$};
  \node (C) [below of=Q, node distance=1.3 cm] {${\mathscr{M}}_{g,n}^{[3]}$};
  \draw[->] (P) to node {} (Q);
  \draw[->] (A) to node {$\Phi_{f_3}$} (C);
  \draw[->] (P) to node [swap] {$f_3$} (A);
    \draw[->] (Q) to node {${\pi}_{g,n}^{[3]}$} (C);
\end{tikzpicture}
$$

If we replace
 $\mathscr{C}_{g,n}^{[3]}$ with
$\mathscr{C}_{g,n}^{+[3]}$ this diagram becomes

 $$
\begin{tikzpicture}[node distance=3.1 cm, auto]
  \node (P) {$V_3^+ := \Phi_{f_3}^*\mathscr{C}_{g,n}^{+[3]}$};
  \node (Q) [right of=P] {${\mathscr{C}}_{g,n}^{+[3]}$};
  \node (A) [below of=P, node distance=1.3 cm] {$C_3$};
  \node (C) [below of=Q, node distance=1.3 cm] {${\mathscr{M}}_{g,n}^{[3]}.$};
  \draw[->] (P) to node {} (Q);
  \draw[->] (A) to node {$\Phi_{f_3}$} (C);
  \draw[->] (P) to node [swap] {$f_3^+$} (A);
    \draw[->] (Q) to node {${\pi}_{g,n}^{+[3]}$} (C);
\end{tikzpicture}
$$

Note that each of the  $n$ sections $s_i^{[3]} : {\mathscr{M}}_{g,n}^{[3]} \to {\mathscr{C}}_{g,n}^{+[3]}$
 induces  a   section  \begin{equation} \label{sections}
s_i^+:= \Phi_{f_3}^*s_i^{[3]}: C_3 \to V_3^+
\end{equation}
of the family $f_3^+:V_3^+ \to C_3$ just defined, and that  $V_3$ equals $V_3^+$ minus the images of these  sections (c.f. \cite[Example 6.2.10]{Hubbard}).

 This new family $V_3^+$ can be seen as the family obtained from
$f_3:V_3 \to C_3$ by compactifying the fibers. In order to also compactify the base, we need to consider
the Deligne-Mumford compactification
$\bar{\mathscr{M}}_{g,n}$  of
 $\mathscr{M}_{g,n}$ introduced by these authors in their foundational article \cite{DM}.
  We recall that $\bar{\mathscr{M}}_{g,n}$ is an irreducible projective variety whose points correspond bijectively to the set of isomorphism classes of stable   curves (or, equivalently, Riemann surfaces with nodes) of type $(g,n)$.
As in the case of the standard   universal curve
there  is a  fibration
$$\bar{\pi}_{g,n}: \bar{\mathscr{C}}_{g,n} \to \bar{\mathscr{M}}_{g,n}$$
	characterized by the property
that the fiber over a point   representing a stable curve   $Y$  of finite type $(g,n)$
is a  stable curve isomorphic to the quotient
 $Y^+/ \mbox{Aut}(Y^+)$, where
  $Y^+$ stands for the compactification of $Y$.
 We shall refer to this map as the
{\it stable $n-$pointed universal curve} of genus $g$.

%

\s

Similarly, the level three universal curve $\pi_{g,n}^{[3]}: {\mathscr{C}}_{g,n}^{[3]} \to {\mathscr{M}}_{g,n}^{[3]}$ gives rise to the
 {\it stable level three $n-$pointed universal curve} of genus $g$
$$\bar{\pi}_{g,n}^{[3]} : \bar{\mathscr{C}}_{g,n}^{[3]} \to \bar{\mathscr{M}}_{g,n}^{[3]}$$
 which is  an extension of the fibration ${\pi}_{g,n}^{+[3]}:\mathscr{C}_{g,n}^{+[3]}
\to \mathscr{M}_{g,n}^{[3]}.$

  We remark that
although
   a  Teichm\"{u}ller theoretic approach to   the
	construction of the compactified moduli spaces  is  also possible
 (see e.g. \cite{Leon} and \cite{Hubbard}), moduli theory of algebraic curves was first built by Mumford and others within the framework of Algebraic Geometry. In particular these moduli spaces, as well as the sections
 $ s_i : {\mathscr{M}}_{g,n} \to {\mathscr{C}}_{g,n}$
 and
 $ s_i ^{+[3]}: {\mathscr{M}}_{g,n}^{[3]} \to {\mathscr{C}}_{g,n}^{+[3]}$,
  are algebraic objects.
Of crucial importance for us will be  that  all these moduli spaces and universal curves, and in particular
the algebraic varieties $\bar{\mathscr{C}}_{g,n}^{[3]},  {\mathscr{C}}_{g,n}^{+[3]}$ and $\bar{\mathscr{M}}_{g,n}^{[3]}$, are even defined over $\mathbb{Q}$ (in fact, over $\mathbb{Z}[\tfrac{1}{3}]$). See \cite{DM} and \cite{Mumford}.

\

We can now construct a suitable compactification of our level three family $f_3:V_3 \to C_3$.
Using a result of Kobayashi on extensions of holomorphic maps between  complex analytic spaces, Imayoshi \cite[p. 289]{ima1} proved that classifying maps of algebraic families of Riemann surfaces can be holomorphically extended to the Deligne-Mumford compactification. In our case, the same arguments employed by Imayoshi ensure that $\Phi_{f_3}$ admits an   extension
\begin{equation} \label{classifcompact}
 \bar{\Phi}_{f_3} : \overline{C}_3 \to \bar{\mathscr{M}}_{g,n}^{[3]}
 \end{equation}
which is is holomorphic and, by Chow's theorem  (see e.g. \cite[Chapter 4]{Dmu}), even algebraic.
Therefore, we can     pull-back   the
  stable level three $n-$pointed universal curve
    to obtain a two-dimensional compact complex analytic space (and, in fact a projective surface)
 \begin{equation} \label{V3compact}
 \overline{V}_3:=\bar{\Phi}_{f_3}^* \bar{\mathscr{C}}_{g,n}^{[3]}
 \end{equation}
 which comes equipped  with a surjection
\begin{equation} \label{familycompactified}
 \overline{V}_3
\xrightarrow{\  \overline{f}_3 \ } \overline{C}_3
\end{equation}
that extends the surjection $f_3:V_3 \to C_3$;  and $n$ sections
 \begin{equation} \label{sectionscompact}
 \overline{s}_i: \overline{C}_3 \to  \overline{V}_3;   \   \    i=1, \ldots, n
\end{equation}
that extend the sections $s_i^+=  C_3 \to V_3^+ \subset \overline{V}_3 $ introduced above   (see e.g. \cite[Chapter I, Proposition 6.8]{Hartshorne}).  We note that $V_3$ is a Zariski open set of $\overline{V}_3$ whose  complement is
   the union of the      stable curves arising as
  fibers of $\overline{f}_3$ over the finite set $\overline{C}_3 \setminus C_3$ together with the images of the sections $\overline{s}_i$.

\section{Graphs of holomorphic motions and
 Uniformization of Families of Riemann Surfaces} \label{s2}
  \label{MH}

In this section we investigate the relationship between holomorphic motions of the unit circle and universal covers of  families of Riemann surfaces. The connection is established via  graphs of holomorphic motions.

\begin{prop} \label{uniquely}
A holomorphic motion of the unit circle $W$ is uniquely determined by its graph.
	\end{prop}
\begin{proof}
Let $\mathscr{B}_1$ and $\mathscr{B}_2$ be
 graphs of holomorphic motions of the circle
  $W^1$ and $W^2$ respectively.
  By the $\lambda-$lemma, the maps $$\theta_i : \Delta \times \Delta \to \mathscr{B}_i \hspace{1 cm} (t,z) \mapsto (t, W_t^i(z))$$
	 are  strong global trivializations of the  associated  families  of quasidisks
	 $ \mathscr{B}_i \to \Delta$ in the sense of Earle and Fowler (see \cite[p. 252]{ERF}).
	
	 Suppose now that the two graphs, hence the two associated families, agree. Then,
	using Lemma 7 in the same paper \cite[p. 263]{ERF}, we deduce that the homotopy class of $$\omega_t:=(W_t^2)^{-1} \circ W_t^1$$
relative to $\mathbb{S}^1$ is independent of $t$. Thus, $\omega_t$ and $\omega_0=(W_0^2)^{-1}\circ W_0^1=id$
 agree on $\mathbb{S}^1;$ and therefore $W_t^1=W_t^2$ on $\mathbb{S}^1$, for all $t \in \Delta$. The proof is done.
\end{proof}

\s

 As the following example shows, Proposition \ref{uniquely} allows us to give plenty of examples of non-trivial holomorphic motions whose graphs are not Bers-Griffiths domains.

\begin{example} \label{counterexample}
The graph
of the holomorphic motion
 $$W(t,z)=z+t^2 \bar{z}$$
  is not a Bers-Griffiths domain.

By Proposition  \ref{uniquely} to check our claim it is enough to show that $W$ is $K$-equivariant for no finite index subgroup $K$ of $K_W$. But, in fact, it is easy to check that it is $K$-equivariant for no  non-trivial group  $K$  of M\"{o}bius transformations.

Indeed, let $\gamma(z)=\lambda \tfrac{z-a}{1-\bar{a}z}$, with $|a| < 1$ and $|\lambda| = 1$. Then,
 $$W_t\circ \gamma \circ W_t^{-1}(z)= \lambda \left[ \frac{(z-t^2\bar{z})-a(1-|t|^4)}{(1-|t|^4)-\bar{a}(z-t^2\bar{z})}+t^2 \frac{(\bar{z}-\bar{t}^2z)-\bar{a}(1-|t|^4)}{(1-|t|^4)-a(\bar{z}-\bar{t}^2z)}  \right]$$
Thus, $\overline{\partial}(W_t\circ \gamma \circ W_t^{-1}) \equiv 0$ if and only if
 $$[(1-|t|^4)-a(\bar{z}-\bar{t}^2z)]^2-\bar{t}^2[(1-|t|^4)-\bar{a}(z-t^2\bar{z})]^2=0$$
for all $t,z \in \Delta.$ Now the claim follows from the observation
 that for $z=0$ this expression equals $(1-|t|^4)^2(1-\bar{t}^2)$ which is always non-zero.

We note that,  according to Wang \cite[p. 6]{Wang}, the graph $ \mathscr{B}_W$ of a holomorphic motion of the form
$W(t,z)=z+a(t) \bar{z}$
	is   biholomorphic to $\Delta^2$
	only when $a\equiv 0$.
	\end{example}

With the same notations as in Subsection \ref{BBA} we have the following:

\begin{prop} \label{p8}
 Let  $W: \Delta \times \Delta \to {\comp}$
be a holomorphic motion
 with associated group sequence
  \[ \xymatrix {
  1 \ar[r] &
  K_{W} \ar[r] &
  \mbox{Aut}_{\pi}(\mathscr{B}) \ar[r] &
  \Gamma_{W} \ar[r] &
  1.
} \]
Suppose that  $W$ is $K_W-$equivariant, then $X_t(K_W)$ agrees with $K_W^t$ for all $t \in \Delta.$ In particular, the groups $K_{W}^t$ are all quasiconformally conjugate to $K_W$ by means of $W_t$.
\end{prop}

\begin{proof}
As $W$ is $K_{W}-$equivariant, for all pairs $(t,z) \in \Delta \times \Delta$ and all M\"{o}bius transformations $\kappa \in K_{W}$, one has $W(t,\kappa(z))=X_t(\kappa)(W(t,z))$ where $X_t(\kappa)$ is a M\"{o}bius transformation; in particular, for each $t \in \Delta$ the map $z \mapsto W_t\kappa W_t^{-1}(z)=X_t(\kappa)(z)$ is holomorphic. We claim that so does the map $t \mapsto W_t\kappa W_t^{-1}(z)$ for each $z \in \Delta.$ Indeed, following \cite[p. 929]{Earle}, let us consider three distinct points $z_1, z_2, z_3$ in $\Delta$ and let\begin{displaymath}
h_t =
\left( \begin{array}{cc}
a_t & b_t\\
c_t & d_t
\end{array} \right), \hspace{1 cm} a_td_t-b_tc_t=1
\end{displaymath}be the unique M\"{o}bius transformation satisfying the property $h_t(z_i)=W(t, z_i)$, for all $i=1,2,3$, so that   $W(t,\kappa(z_i))=X_t(\kappa)(h_t(z_i)).$ Then, as the maps $t \mapsto W(t, z_i)$ and $t \mapsto W(t,\kappa(z_i))$ are holomorphic, we see that $t \mapsto h_t$ and $t \mapsto X_t(\kappa) \circ h_t$ are holomorphic as well. Thereby, the rule $t \mapsto X_t(\kappa(z))=W_t\kappa W_t^{-1}(z)$ defines a holomorphic map, as claimed.

We conclude that for all $\kappa \in K_{W}$ the rule $(t,z) \mapsto (t, W_t\kappa W_t^{-1}(z))$ defines a biholomorphism of $\mathscr{B}_W$ lying in the kernel group $\mathbb{K}_{W}.$ Now, as $\Phi_0: \mathbb{K}_{W} \to K_{W}$ is an isomorphism, this biholomorphism must agree with $\Phi_0^{-1}(\kappa).$ It follows that each element of $\mathbb{K}_{W}$ is of the form $(t,z) \mapsto (t,W_t\kappa W_t^{-1}(z) )$ for some $\kappa \in K_{W}$ and therefore $X_t(K_W)=K_{W}^t$ for all $t \in \Delta.$ The last statement follows from the $\lambda-$lemma.
\end{proof}

It is well known that there are only three simply connected Riemann surfaces. In contrast, understanding universal covers of higher dimensional complex analytic manifolds  seems to be far more complicated. However, thanks to the work of Bers \cite{Bers} and Griffiths \cite{Griffiths} on uniformization of algebraic varieties, it is possible to describe the universal cover of an algebraic family of Riemann surfaces in the following very explicit form.

\s

\begin{prop} \label{p9}
Let $f: V \to C$ be a non-isotrivial algebraic family of Riemann surfaces.
Let $p : \Delta \to C$ be the universal covering map of $C$ with covering group $\Gamma$ and let $h: p^*V \to \Delta$ be the corresponding pull-back family.
  Let us denote  by $X\equiv \Delta/K $ the central fiber $h^{-1}(0) \cong f^{-1}(p(0))$ and by $\tilde{\Phi}:\Delta \to T(X)\equiv T(K)$ the classifying map (\ref{classifyingmap}). Then, the
	 universal cover of
 $f: V \to C$
	is isomorphic to the
 Bers-Griffiths family  $\mathscr{B}_W \to \Delta$ defined by the
   $K$-equivariant holomorphic motion $W$ given by
$$
 W: \Delta \times \Delta \to  \comp \hspace{1 cm} (t,z) \mapsto W(t,z):=w^{\mu_t}(z)
$$
 where the map $t\in \Delta \to \mu_t \in L^{\infty}_1(\Delta,K)$ results from the composition of  $\tilde{\Phi}$ followed by a fixed continuous section (e.g. the Douady-Earle section) of
$F(K)\to T(K)$.
 \end{prop}
 \begin{proof}
This proposition is  a consequence of Bers'  results  on universal families and classifying maps. It is stated in this same form in Theorem 3.1 of \cite{Chirka}; but it can be easily derived from the properties of the classifying map described in Section \ref{s12}.
  Clearly, the universal cover of $V$ agrees with the universal cover of
$p^*V \cong (\tilde{\Phi})^*V_{g,n}$  (see (\ref{pullback})), and the latter
 can be realized as the pull-back of the Bers fiber space, namely (see (\ref{BersFiberspace})) $$(\tilde{\Phi})^*F_{g,n}=\{(t, z) : t \in \Delta, z \in w^{\mu_t}(\Delta)\}$$
with    $[\mu_t] =\tilde{\Phi}(t).$
   We note   that $w^{\mu_0}(z)=z$ for all $z \in \Delta.$	 	
\end{proof}	

 \

We end this section by turning our attention to the automorphism group of Bers-Griffiths domains.
Let  $\mathbb{G} < \mbox{Aut}_{\pi}(\mathscr{B}_W)$  be the universal covering group of our algebraic family $f:V\to C.$
 If we restrict to $\mathbb{G}$
the sequence associated to $W$  we obtain a new exact sequence   \[ \xymatrix {
  1 \ar[r] &
  \mathbb{K}:=\mathbb{K}_W \cap \mathbb{G} \ar[r] &
  \mathbb{G} \ar[r]^{\Theta \,\,\,\,\,\,\,\,\,\,\,\,\,} &
  \Gamma:=\Theta(\mathbb{G}) \ar[r] &
  1
} \]
which shall be referred to as {\it the group sequence
associated to the the family} $f:V\to C.$
We observe that the family   can be completely recovered from its associated sequence as $\mathscr{B}/\mathbb{G} \to \Delta /\Gamma \equiv C $; the fiber over a point $ [t]_{\Gamma}$ being the Riemann surface  $D_t/K^t$ , where
$K^t:=\Phi_t( \mathbb{K})$
is the restriction of the subgroup $\mathbb{K}<\mathbb{K}_W$ to the quasidisk $D_t$, which by Proposition \ref{p8}
agrees with
$W_tK^0W_t^{-1}$. We note that the groups $K^0$  and $\Gamma$ are both torsion free Fuchsian groups.

 \

In \cite{Sha} and \cite{Shabat}
Shabat studied the automorphism group of the universal cover $\mathscr{B}$ of an algebraic family of Riemann surfaces. He showed that except for the case in which $\mathscr{B}$ is a bounded homogeneous domain in $\comp^2$, the group
$\mbox{Aut}(\mathscr{B})$ is  discrete.
By a well-known result of Cartan there are only
two such exceptional domains, namely the unit ball and the  bidisk.
  However, the first one  never occurs as the universal cover of an algebraic family of Riemann surfaces (see \cite[Theorem 1]{IN}) and the latter arises only  when the family is isotrivial
 (see \cite[Theorem 2]{IN}). This being observed,
Shabat's results can be formulated as follows (see also \cite{BB} and \cite{AA} for more details).

\s

{\bf Theorem (Shabat).}
Let $f: V \to C$ be a non-isotrivial algebraic family of Riemann surfaces and $\mathscr{B}$ the universal cover of $V.$ Then:
\begin{enumerate}
\item[(a)] $\mbox{Aut}(\mathscr{B})$ acts properly discontinuously on $\mathscr{B}.$ \item[(b)] The universal covering group of $V$ has finite index in $\mbox{Aut}(\mathscr{B}).$
\end{enumerate}


\section{Proof of Theorem \ref{theo0}}
\label{s4}

(1) Let $(\mathscr{B},\pi)$ be the universal cover of an algebraic family. Then, with the notation of Proposition \ref{p9}, Shabat's Theorem implies that  the base and
the fiber groups $K_W$ and $\Gamma_W$ of the holomorphic motion $W$ contain the Fuchsian groups
$K$ and $\Gamma$ as finite index subgroups; therefore $K_W$ and $\Gamma_W$ must be Fuchsian as well. The $K-$equivariancy was already established in the same proposition. This proves that $(\mathscr{B}, \pi)$ is isomorphic to a Bers-Griffiths family.

\s

(2) Conversely, let us suppose that $\mathscr{B}_W \to \Delta$ is a Bers-Griffiths family. This means that   $\Gamma_{W}$ and $K_{W}$ are Fuchsian groups of finite hyperbolic type and that   $W$   is $K-$invariant for some finite index subgroup  $K$ of $K_{W}$. In fact, by    Lemma \ref{andrei}  below it can be assumed that   $K=\mathbb{G} \cap K_{W}$, for some finite index subgroup $\mathbb{G}< \mbox{Aut}_{\pi}(\mathscr{B}_W).$
 What we have to do is to construct an algebraic family of Riemann surfaces $V\to C$ of which $\mathscr{B}_W \to \Delta$ is the universal cover.

By the work of Earle, Kra and Krushkal (see \cite[Theorem 1]{Earle})
we can extend $W$ to a $K-$equivariant holomorphic motion of $\Delta$ which we still denote by $W$.

Let
 \[ \xymatrix {
  1 \ar[r] &
  \mathbb{K}:=\mathbb{K}_{W} \cap \mathbb{G} \ar[r] &
  \mathbb{G} \ar[r] &
  \Theta(\mathbb{G}):=\Gamma  \ar[r] &
  1
} \]
be the exact sequence obtained by restriction of $\Theta: \mbox{Aut}_{\pi}(\mathscr{B}_W) \to \Gamma_W$ to $\mathbb{G}.$ By Selberg's lemma \cite{Selberg} there is a finite index normal subgroup $\Gamma' $ of $\Gamma $ which is torsion free. By further  restriction of $\Theta$ to $\mathbb{G}':=\Theta^{-1}(\Gamma')$ one obtains a sequence \[ \xymatrix {
  1 \ar[r] &
  \mathbb{K}'=\mathbb{K} \ar[r] &
  \mathbb{G}' \ar[r] &
  \Gamma' \ar[r] &
  1
} \] with the same kernel. We shall denote by $K^t$ the image of $\mathbb{K}$ under  the monomorphism $\Phi_t: \mathbb{K}_{W} \to \mbox{Aut}(D_t)$ defined in (\ref{restriction}). We remark that $K^t$ is a finite index subgroup of $K_{W}^t$ for all $t \in \Delta.$ (Note that
$K^0=K.$)

\s

{\bf Claim.} $\mathbb{G}'$ acts properly discontinuously on $\mathscr{B}=\mathscr{B}_W.$

\s

Let $(t', z')$ be an arbitrary point of $\mathscr{B}.$ As $\Gamma' $ is a torsion free Fuchsian group, there exists a neighborhood $U \subset \Delta$ of $t'$ such that $\gamma(U) \cap U = \emptyset$ for each non-trivial element $\gamma \in \Gamma'.$ Let $z_0$ be the point in $\Delta$ such that $W_{t'}(z_0)=z'.$ As we are assuming that $K_{W}$ is a Fuchsian group, the group  $K$ is Fuchsian as well. So, there exists a neighborhood $V_0 \subset \Delta$ of $(0,z_0)$ such that $\kappa(V_0) \cap V_0 \neq \emptyset$ for only finitely many elements $\kappa \in K$. Let us now consider the map
 $$\phi_W : \Delta \times \Delta \to \Delta \times \comp \hspace{1 cm} (t,z) \mapsto (t, W_t(z))$$
whose image is $\mathscr{B}.$
By the  $\lambda$-lemma   holomorphic motions are continuous, thus the map  $\phi_W$ is a continuous injection.
 Hence, the Invariance of  Domain Theorem allows us to assert that $$\mathscr{U}:=\phi_W(U \times V_0) =\{(t,z): t \in U, z \in W_t(V_0)\} \subset \mathscr{B}$$is a neighborhood of $(t',z')$ homeomorphic to $U \times V_0.$

Now, let $g(t,z)=(\hat{g}(t), g_t(z))$ be an automorphism of $\mathscr{B}$ in $\mathbb{G}'$ such that $g(\mathscr{U}) \cap \mathscr{U} \neq \emptyset.$ Then, there are points $(t_1, z_1)$ and $(t_2, z_2)$ in $\mathscr{U}$ in such a way that $$g(t_1,z_1)=(\hat{g}(t_1), g_{t_1}(z_1))=(t_2,z_2).$$
As $\Theta(g)=\hat{g} \in \Gamma'$ the choice of $\mathscr{U}$ implies that $\hat{g}=id$. We conclude that
 $g\in \mathbb{K}$,
$t_1=t_2,$  and $g_{t_1}\in K^{t_1}.$
Let $z_0^i$ be the unique point in $V_0$ such that $W_{t_1}(z_0^i)=z_i, (i=1,2).$
By invoking the same arguments used in the proof of Proposition \ref{p8},
we see that
 $K^{t_1}=W_{t_1}KW_{t_1}^{-1}.$ Thus, there is some $\kappa \in K$ such that $g_{t_1}\in W_{t_1}\kappa W_{t_1}^{-1}.$ We observe that $\kappa(z_0^1)=z_0^2$ which means that $\kappa(V_0) \cap V_0 \neq \emptyset$ and then, by the choice of $V_0,$ there are only finitely many possibilities for $\kappa.$ Consequently, there are also only finitely many possibilities for $g_{t_1},$ hence for $g$ (see Subsection \ref{BBA}). This proves the claim.

\s

Our claim implies that the sequence\[ \xymatrix {
  1 \ar[r] &
  \mathbb{K}'=\mathbb{K}  \ar[r] &
  \mathbb{G}' \ar[r] &
  \Gamma' \ar[r] &
  1
} \]
yields a holomorphic map $$V:= \mathscr{B}/\mathbb{G}' \to C:=\Delta/\Gamma' $$
between a two-dimensional complex analytic space $V$ and a Riemann surface $C$ of finite hyperbolic type,
with fibers homeomorphic  $\Delta/K.$
Now, if
  $\mathbb{G}'$ acted freely on $\mathscr{B}$ then
we could conclude that $V \to C$ is a non-isotrivial   algebraic  family of Riemann surfaces
(with horizontally holomorphic trivializations provided by the holomorphic motion $W$).
In any case, as $\Gamma'$ is torsion free, the elements of $\mathbb{G}'$ that fix some point must belong to the kernel group $\mathbb{K}'$ and be in bijection with the torsion elements of $K$. Therefore there is only a finite number of conjugacy classes of them. Now, by a result of Johnson
  \cite[Theorem 3.6]{residual},
  the facts that $\Gamma'$ is a torsion free Fuchsian group and that $\mathbb{K}' $ is finitely generated and residually finite, imply that $\mathbb{G}'$ is a residually finite group; (alternatively, see \cite[ Corollary 3.7]{Andrei}).
  It follows that there exists a finite index normal subgroup $\mathbb{G}''$ of $\mathbb{G}'$ which does not contain any non-trivial torsion element. Now, to obtain the result, it  only remains to replace
$\mathbb{G}'$
by
$\mathbb{G}''.$ The proof of the theorem is now complete.

\begin{lemm} \label{andrei}
Let
\[ \xymatrix {
  1 \ar[r] &
  K \ar[r] &
  \mathbb{A} \ar[r] &
  \Gamma \ar[r] &
  1
} \]
be an exact sequence of abstract groups  in which  $K$ and $\Gamma$ are isomorphic to Fuchsian groups of finite type.  Then, for every finite index subgroup $N$ of $K$ there is a finite index subgroup $\mathbb{G}$ of $\mathbb{A}$ such that the group   $  K \cap \mathbb{G}$ is a finite index subgroup of $N$.
\end{lemm}
\begin{proof} The proof can be achieved through the following four steps:
 \begin{enumerate}
\item[(1)] Using Selberg's lemma, as above, we can assume that  $\Gamma$ is torsion free, hence either free or a surface group. In particular $\Gamma$ is good in the sense of Serre (see \cite{Serre}).
\item[(2)]   Replacing $N$ by the intersection of all subgroups of $K$ of same index as $N$ we can assume that $N$ is normal in $\mathbb{A}$.
 \item[(3)] Considering the exact sequence
 \[ \xymatrix {
  1 \ar[r] &
  K/N \ar[r] &
  \mathbb{A}/N \ar[r] &
  \Gamma \ar[r] &
  1
} \]
the problem is reduced to show that there is a finite index   subgroup $\mathbb{G}/N$ of $\mathbb{A}/N$ such that  $(K/N) \cap (\mathbb{G}/N) $ is the trivial group, i.e. such that   $\mathbb{G}/N$
 does not contain any of the non-trivial elements of the finite group  $K/N$. This will of course be the case if the group $\mathbb{A}/N$ were residually finite.
 \item[(4)]  The desired conclusion is reached by applying the following result due to Grunewald, Jaikin-Zapirain and Zaleski (\cite[Proposition 3.6]{Andrei}):

 Let $\Gamma$ be a residually finite good group and $\varphi: A \to \Gamma$ a surjective homomorphism with finite kernel. Then $A$ is residually finite.
 \end{enumerate}
\end{proof}

\begin{rema} \label{sulg} For later use, we observe that
 in proving the theorem, we have exhibited two ways to construct new families of Riemann surfaces of arbitrarily high base and  fiber genus out of a given one.
  Namely, if the given family has associated sequence
   \[ \xymatrix {
  1 \ar[r] &
  \mathbb{K} \ar[r] &
  \mathbb{G} \ar[r] &
  \Gamma \ar[r] & 1} \]
then, by considering its restriction to a finite index subgroup of $\Gamma$ or/and
 $\mathbb{G}$
  we were able to obtain a new family whose base, and by Lemma \ref{andrei} also its fibers, has higher genus. The importance of this simple remark for us lies on the fact  that all these families have the same
universal covers.
\end{rema}

\section{Proof of Theorem \ref{coro} } \label{sec4}

(1) ({\bf The only if part})

Suppose that $(\mathscr{B}, \pi)$ is the universal cover of an arithmetic family
 $f: V \to C$. By Theorem  \ref{theo0}, the family   $(\mathscr{B}, \pi)$ is isomorphic to
   a Bers-Griffiths family
 $\mathscr{B}_W \to \Delta.$  We must show that
  $\mathcal{O}_{W}$ is an arithmetic orbifold.

 Let     $\mathbb{G} < \mbox{Aut}_{\pi}(\mathscr{B})$ the universal cover of $V$ and
 \[ \xymatrix {
  1 \ar[r] &
  \mathbb{K} \ar[r] &
  \mathbb{G} \ar[r] &
  \Gamma \ar[r] &
  1
} \]
be the group sequence
associated to the  family $f: V \to C$.

 By Shabat's Theorem $\Gamma$ is a finite index subgroup of $\Gamma_{W}$ and therefore the inclusion $\Gamma \le \Gamma_{W}$ gives rise to a finite degree branched covering map between $C\cong \Delta/\Gamma$  and $R\cong \Delta/\Gamma_W.$
We shall denote by $\beta: \overline{C} \to \overline{R}$ the induced covering map between the corresponding compact Riemann surfaces. We note that, being the image of a curve defined over $\hola$ under  a covering map, the algebraic curve $\overline{R}$ is also defined over $\hola$. \cite[Theorem 4.4]{Gabino2}.

Now, let $\gamma_1, \ldots, \gamma_s$ be representatives of the cosets of $\Gamma_{W}/\Gamma.$ Consider the normal core subgroup $\Gamma_0:=\cap_{i=1}^s \gamma_i \Gamma \gamma_i^{-1}$  of $\Gamma.$ We notice that $\Gamma_0$ is a torsion free finite index normal subgroup of $\Gamma_{W}.$ We denote by $C_0$ the Riemann surface $\Delta/\Gamma_0$ and by $\pi_1 : C_0 \to C$ and $\pi_0: C_0 \to R$ the covering maps induced by the inclusions $\Gamma_0 \le \Gamma$ and $\Gamma_0 \le \Gamma_{W}.$  The maps $\pi_1$ and $\pi_0$ extend to covering maps $\bar{\pi}_1$ and $\bar{\pi}_0$ between the respective compactifications in such a way that the following diagram commutes$$
\begin{tikzpicture}[node distance=1.5 cm, auto]
  \node (P) {$\overline{C}_0$};
  \node (A) [below of=P, left of=P] {$\overline{C}$};
  \node (C) [below of=B, right of=P] {$\overline{R}.$};
   \draw[->] (P) to node [swap] {$\bar{\pi}_1$} (A);
  \draw[->] (P) to node {$\bar{\pi}_0$} (C);
  \draw[->] (A) to node [swap] {$\beta$} (C);
\end{tikzpicture}
$$
As $C$ is an arithmetic curve the branch \label{ver}values of $\bar{\pi}_1,$ which are contained in $\overline{C} \setminus C,$ are points defined over $\hola$. It follows that both the curve $\overline{C}_0$ and the map $\bar{\pi}_1$ are defined over $\hola$ \cite[Theorem 4.1]{Gabino2}. Moreover, since the arithmeticity of a hyperbolic algebraic curve implies that of its automorphisms  \cite[Corollary 3.4]{Gabino2}, it follows that the normal covering  $\bar{\pi}_0:\overline{C}_0 \to \overline{R}$
is arithmetic as well.
Now, since by construction  $\overline{R} \setminus R$ is the image under
 $\bar{\pi}_0$ of the preimage under $\bar{\pi}_1$ of the
   arithmetically defined set  $\overline{C} \setminus C$,
  the points in $\overline{R} \setminus R$ are also defined over $\hola$; hence $R$ is arithmetic. As, in addition,  each branch value of the universal covering map $\Delta \to R$ is a  branch value of $\bar{\pi}_0,$ we  conclude that each  conic point $q_i$ of the orbifold $\mathcal{O}_{W}$ is also defined over $\hola.$ This proves that $\mathcal{O}_{W}$ is an arithmetic orbifold, as wanted.

\s

(2) Before starting with the proof of the converse we need  to establish a couple of facts.
Let
$\Phi : C \to \mathscr{M}_{g,n}$  be the classifying map of an algebraic  family of Riemann surfaces $f:V\to C$.
In  Section \ref{s12},
we introduced a covering $C_3 \to C$, a
 new family
$f_3 : V_3 \to C_3$,  with same universal cover as $f:V\to C$, and a compactification
$\overline{f}_3 : \overline{V}_3 \to \overline{C}_3$ such that $\overline{V}_3$ is a projective surface tha
contains $ V_3$ as a Zariski open set. The
 complement
 $\overline{V}_3 \setminus V_3$ consists of
  a  finite number of   stable curves arising as
  fibers of $\overline{f}_3$ over the points in $\overline{C}_3 \setminus C_3$ together with the images of $n$ sections $\overline{s}_i:\overline{C}_3 \to \overline{V}_3 $  (see (\ref{sectionscompact})). Moreover, by Remark \ref{sulg},  we can assume that the base and the generic fiber of $\overline{f}_3 : \overline{V}_3 \to \overline{C}_3$
 are compact Riemann surfaces of genus greater or equal to two.

\s

{\bf Claim 1.}  $\overline{V}_3 $ is  of general type.

\s

 Let us consider the
fibration
$$\tilde{f}_3 : \tilde{V}_3 \to \overline{C}_3$$
  obtained by passing to a resolution of singularities   $\tilde{V}_3$ of $\overline{V}_3.$ Since
   $\tilde{f}_3$
   is  a morphism between smooth  projective varieties, the sub-additivity of the Kodaira dimension (see  e.g. \cite[Theorem 2]{Kaw}) implies that the projective surface $\tilde{V}_3$ is of general type, and therefore so must be $\overline{V}_3$.


\s

{\bf Claim 2.} If $\sigma \in \mbox{Gal}(\comp)$ and $\Psi : \bar{\mathscr{M}}_{g,n}^{[3]} \to \bar{\mathscr{M}}_{g,n}$ is the finite degree holomorphic map given by forgetting the level structure, then
$$\Psi \circ \bar{\Phi}_{f_3}^{\sigma}=\Psi \circ \bar{\Phi}_{f_3^{\sigma}}$$
where $\bar{\Phi}_{f_3}:\overline{C}_3 \to \bar{\mathscr{M}}_{g,n}$ (see (\ref{classifcompact}))
is the closure of the classifying map of the family  $f_3:V_3 \to  C_3$ and, of course,
$\bar{\Phi}_{f_3}^{\sigma}$ stands for  $(\bar{\Phi}_{f_3})^{\sigma} .$
In particular, the map $\bar{\Phi}_{f_3}^{\sigma}$ is determined by the map $\Phi_{f_3^{\sigma}}$ up to finitely many options.

\s

If the equality holds in $C_3^{\sigma}$ then it will also hold in $\overline{C_3^{\sigma}}$ by continuity. Now, if $y=\sigma(x) \in C_3^{\sigma}$  then $$\Psi \circ (\bar{\Phi}_{f_3})^{\sigma}(\sigma(x))=[(f_3^{-1}(x))^{\sigma}]=[(f_3^{\sigma})^{-1}(\sigma(x))] =  \Psi \circ \Phi_{f_3^{\sigma}}(\sigma(x)).$$

\s

(3) {(\bf The if part)}

We are  in position to begin the proof of the converse. We are now assuming that
$(\mathscr{B}, \pi)$ is isomorphic to a Bers-Griffiths family $\mathscr{B}_W \to \Delta$ such that $\mathcal{O}_{W}$ is an arithmetic orbifold. By Theorem \ref{theo0}
this implies that    $(\mathscr{B}, \pi)$  is the universal covering of an algebraic family $f : V \to C$ and   hence of the associated three level family $f_3 : V_3 \to C_3$ mentionned above.
What we have to prove is that, because  $\mathcal{O}_{W}$ is  arithmetic, this family is arithmetic.

Now, the inclusion
of the group $\Gamma_3$ that uniformizes the curve $C_3$ (see Section \ref{s12}) in
 $  \Gamma_{W}$ yields a branched covering map between compact Riemann surfaces which we denote by $\beta: \overline{C}_3 \to \overline{R}.$
 Since our hypothesis implies that
  the algebraic curve   $\overline{R}$, the branch values of $\beta$
  and the set
   $\overline{R} \setminus R$ are defined over $\hola$,  arguing as in part $(1),$
   we can assert that the
     covering $\beta: \overline{C}_3 \to \overline{R}$ and hence
 the points in $\overline{C}_3 \setminus C_3$ are also defined over $\hola.$
   It follows that $C_3$ is arithmetic.

   Therefore, the collection $\{C_3^{\sigma}\}, \sigma \in \mbox{Gal}(\comp)$ contains only  finitely many isomorphism classes of curves.
   By Arakelov's Finiteness Theorem (see \cite[p. 207]{ShiIma})
    this implies that
  there are only finitely many non-isotrivial non-equivalent families of finite hyperbolic type $(g,n)$ over $C_3^{\sigma}$, for each $\sigma \in \mbox{Gal}(\comp).$
   In particular, the set of isomorphism classes of algebraic families
  $ \{ f_3^{\sigma}: V_3^{\sigma} \to C_3^{\sigma} :   \sigma \in \mbox{Gal}(\comp)\}$,
   and hence the set of maps
   $ \{ \Phi_{f_3^{\sigma}}  : C_3 \to \mathscr{M}_{g,n}; \sigma \in \mbox{Gal}(\comp)\}$,
  is finite.

   Furthermore, as $\bar{\mathscr{C}}_{g,n}^{[3]}$ is a projective variety defined over $\mathbb{Q},$ we can write
  $$
  (\overline{V}_3)^{\sigma}=(\bar{\Phi}_{f_3}^*\bar{\mathscr{C}}_{g,n}^{[3]})^{\sigma}= (\bar{\Phi}_{f_3}^{\sigma})^* (\bar{\mathscr{C}}_{g,n}^{[3]})^{\sigma}= (\bar{\Phi}_{f_3}^{\sigma})^* \bar{\mathscr{C}}_{g,n}^{[3]}.
  $$
  Now, Claim 2 allows us to state that the collection
\begin{equation}   \label{Finitefamily}
   \{ (\overline{V}_3)^{\sigma} : \sigma \in \mbox{Gal}(\comp)\}
 \end{equation}
   also contains only finitely many isomorphism classes.

Let now
 $$\pi_3^{ \tiny \mbox{nor}} : (\overline{V}_3)^{ \tiny \mbox{nor}} \to \overline{V}_3$$
 be the normalization of
 $ \overline{V}_3 $ (see \cite[p. 25]{Barth}) and
$$\pi_3^{ \tiny \mbox{des}} : Z \to (\overline{V}_3)^{ \tiny \mbox{nor}}$$
 be the
  {\it minimal
 resolution of singularities} of   $(\overline{V}_3)^{ \tiny \mbox{nor}}$
 (see \cite[p. 86]{Barth}).
 We recall that these two
 maps  are
  uniquely determined up to isomorphism.

\s




\s

Now, since the families $V_3$ and
$V_3^+$  (obtained from $V_3$ by filling in the punctures; see Section \ref{s12})
 are subvarieties of
$\overline{V}_3$
containing no singular points, the composed map $$ \lambda := {\pi}_3^{ \tiny \mbox{nor}}\circ  \pi_3^{ \tiny \mbox{des}} : Z \to  \overline{V}_3$$
	induces isomorphisms between $V_3$ and  $U_3 := \lambda^{-1}(V_3)$
		and 	between	$V_3^+$ and
$U_3^+:= \lambda^{-1}(V_3^+).$ In these terms what remains to be shown is that the family  $f_3 \circ \lambda : U_3 \to C_3$
 is arithmetic.

In order to do that we  observe that
the the extended map
$$h:=\overline{f}_3 \circ \lambda : Z \to \overline{C}_3$$
is a fibration  whose generic fiber is a compact Riemann surface of genus at least two.  Therefore, by Arakelov's   Theorem (\cite{Ara},
see also \cite[Theorem 3.1]{Caporaso}),   the collection $$\{h^{\sigma} : Z^{\sigma} \to \overline{C}_3^{\sigma}: \sigma \in \mbox{Gal}(\comp)\}$$
contains finitely many isomorphism classes of fibrations.
Hence,  without lost of generality, we can suppose that $Z,$ $\overline{C}_3$ and $h$ are defined over $\overline{\rac}$ \cite[Criterion 2.2]{Gabino2}.
Thus, our problem is reduced to proving that the Zariski closed
 subset $Z \setminus U_3$ is defined over $\overline{\rac}.$

 Now  $Z\setminus U_3$ agrees with the union of the preimage under $h$ of the finite set $\overline{C}_3 \setminus C_3$
and the image
of the $n$ sections $s_i: C_3 \to U_3 ^+$ (see (\ref{sections})) or rather of their extensions
$\overline{s}_i: \overline{C}_3 \to Z$
(see e.g.
	\cite[Proposition 6.8]{Hartshorne}).
	The first of these two sets is
	defined over $\overline{\rac}$ because both the map $h$ and the set
	$\overline{C}_3 \setminus C_3$ are defined over $\overline{\rac}$. We claim that the second one also is. Indeed, if $\sigma \in \mbox{Gal}(\comp/\overline{\rac})$ then
	$$h \circ s_i^{\sigma}=(h \circ s_i)^{\sigma}=id^{\sigma}=id.$$
This shows that all the maps $\{s_i^{\sigma}\}_{\sigma}$ are sections. But there are only finitely many of them \cite{ShiIma}.  We deduce that each of the sections  $s_i$ is defined  over $\overline{\rac}$ \cite[Criterion 2.2]{Gabino2} which proves    our claim. This bring the proof to an end.

\s


\s

A case in which Theorem \ref{coro}
applies is the following one.
\begin{coro} \label{coro3}

A  Bers-Griffiths family $\mathscr{B}_{W} \to \Delta$  defined by a holomorphic motion whose associated base group $\Gamma_{W}$
is commensurable to a hyperbolic triangle group must be of arithmetic type.
 \end{coro}
\begin{proof}
Let us assume that $\Gamma_{W}$ is commensurable to a hyperbolic triangle group $\Gamma_{abc}.$ Let $N$ be a common finite index subgroup of $\Gamma_{W}$ and $\Gamma_{abc}.$ By Selberg's lemma \cite{Selberg} we can suppose that $N$ is torsion free. Moreover, by passing to the normal core subgroup of $N$, we can suppose that $N$ itself is a normal subgroup of $\Gamma_{W}.$ Let us denote by $C$ the Riemann surface $\Delta/N$ and by $\overline{C}$ its compactification. Now, $\Delta/\Gamma_{abc}$ is isomorphic to $\mathbb{P}^1 \setminus \Sigma$ where $\Sigma$ is a subset of $\{\infty, 0, 1\}$ whose cardinality agrees with the number of integers $a,b,c$ that equal $\infty.$ Moreover, the inclusion $N < \Gamma_{abc}$ induces a Belyi function $ \beta: \overline{C} \to \mathbb{P}^1$ such that $C=\overline{C} \setminus \beta^{-1}(\Sigma).$ Thus $\overline{C}$ is an arithmetic curve (see for example \cite{libroge}). As Belyi functions are arithmetic, we see that the points in $\overline{C} \setminus C$ are arithmetic as well. Now by considering the map induced by the inclusion $N < \Gamma_{W}$ we can argue similarly as done in the proof of Theorem \ref{coro} to conclude that $\mathcal{O}_{W}$ is an arithmetic orbifold.
\end{proof}	
	
%

\section{Proof of Theorems \ref{aritmetica} and \ref{aritmeticaK} } \label{s7}

Let us suppose that $S$ contains a Zariski open subset $U$ whose universal cover $\mathscr{B}$ is isomorphic to
 a Bers-Griffiths domain
  of arithmetic type. Then, by Theorem \ref{coro},
  there is a non-isotrivial  arithmetic family $f: V \to C$ which has  $(\mathscr{B},\pi)$ as its universal cover, for some suitable projection map $\pi: \mathscr{B} \to \Delta.$ Our task is to show that $S$ is arithmetic.

  Replacing $f: V \to C$  by  the associated level three family  $f_3: V_3 \to C_3$ as we did in the  proof of Theorem  \ref{coro}, we can assume that the base and the fibers of our family  have genus at least  two   and that
  (the Zariski closure of) $V$ is a surface of general type
   (see Claim 1 in the proof of Proposition \ref{coro}).

Let \[ \xymatrix {
  1 \ar[r] &
  \mathbb{K}_{2} \ar[r] &
  \mathbb{G}_2 \ar[r]^{\,\Theta} &
  \Gamma_2 \ar[r] &
  1
} \]
be the sequence associated to $f$ and let $\mathbb{G}_1$ be a subgroup of $\mbox{Aut}(\mathscr{B})$ such that $U \cong \mathscr{B}/\mathbb{G}_1.$ We denote by $\mathbb{G}_{12}$ the intersection group $\mathbb{G}_1 \cap \mathbb{G}_2,$ by  $V'$ the respective quotient complex surface $\mathscr{B}/\mathbb{G}_{12}$ and by $p_1$ and $p_2$ the holomorphic maps induced by the inclusion $\mathbb{G}_{12} \le \mathbb{G}_1$ and $\mathbb{G}_{12} \le \mathbb{G}_2$ respectively. Thus, we have

$$
\begin{tikzpicture}[node distance=1.3 cm, auto]
  \node (P) {$V'$};
  \node (A) [below of=P, left of=P] {$U$};
  \node (C) [below of=B, right of=P] {$V$};
  \draw[->] (P) to node [swap] {${p}_1$} (A);
  \draw[->] (P) to node {${p}_2$} (C);
\end{tikzpicture}
$$

{\bf Claim.} $p_1: V' \to U$ and $p_2: V' \to V$ are finite degree holomorphic maps between quasiprojective surfaces.

\s

By Shabat's Theorem, the group $\mathbb{G}_2$ has finite index in $\mbox{Aut}(\mathscr{B})$ and therefore $\mathbb{G}_{12}$ has finite index in $\mathbb{G}_1.$ Thus, $p_1$ is a holomorphic map of finite degree whose image $U$ is a quasiprojective surface. Therefore,
Riemann Existence's Theorem (see \cite[p. 227]{Existencia}) allows us to conclude that $V'$ has a unique structure of quasiprojective surface in such a way that $p_1$ is a morphism; this proves the claim for $p_1.$

Furthermore, as $V$ is of general type,
a result of Kobayashi (see \cite[p. 370]{Koba2}) implies that $p_2:V' \to V$ can be seen as a meromorphic map between the Zariski closures of $V'$ and  $V$ in their respective projective spaces. By
Chow's Theorem (see e.g. \cite[Chapter 4]{Dmu}) this in turn  implies that
 $p_2:V'\to V$   is a regular map; and
therefore, there is a Zariski open subset  of $V$  in which each point has a finite number of preimages (see, e.g. \cite[p. 46]{Dmu}). This proves the claim.

\s

Now, the restriction of the homomorphism $\Theta: \mathbb{G}_2 \to \Gamma_2$ to $\mathbb{G}_{12}$ gives rise to a sequence \[ \xymatrix {
  1 \ar[r] &
  \mathbb{K}_{2} \cap \mathbb{G}_{12} \ar[r] &
  \mathbb{G}_{12} \ar[r]&
 \Theta(\mathbb{G}_{12}) \ar[r] &
  1,
} \]
which induces a new non-isotrivial algebraic family of Riemann surfaces; we shall denote this family by $f': V' \to C'.$

As we have done at the beginning of this proof, we now consider the
level three family
 associated to the family
  $f': V' \to C'$ and we denote it by $f'_{3}: V'_{3} \to C'_{3}$. This family fits into
   the following commutative diagram
  $$
\begin{tikzpicture}[node distance=3.0 cm, auto]
  \node (P) {$V$};
  \node (Q) [right of=P] {$V'$};
    \node (R) [right of=Q] {$V'_{3}$};
  \node (A) [below of=P, node distance=1.2 cm] {$C$};
  \node (C) [below of=Q, node distance=1.2 cm] {$C'$};
   \node (S) [below of=R, node distance=1.2 cm] {$C' _{3}$};
  \draw[->] (Q) to node {} (P);
  \draw[->] (C) to node [swap]{$$} (A);
  \draw[->] (P) to node [swap] {$f$} (A);
    \draw[->] (Q) to node {$f'$} (C);
        \draw[->] (S) to node {$$} (C);
        \draw[->] (R) to node {$$} (Q);
        \draw[->] (R) to node {$f'_3$} (S);

\end{tikzpicture}
$$
where the horizontal rows are unbranched covering maps of finite degree.


We now denote by
$\overline{f_{3}'}: \overline{V'_{3}} \to \overline{C'_{3}}$ the fibration that naturally compactifies the family $f'_{3}: V'_{3} \to C'_{3}$ (see   (\ref{familycompactified}) in Section \ref{s12}).
 As we are assuming that $C$ is an arithmetic curve, we can argue as in the  third part of the proof of Theorem  \ref{coro} to conclude that so are the curves $C'$ and
 $C'_3$  and, ultimately, that the family
 $$\{ (\overline{V'_3})^{\sigma} : \sigma \in \mbox{Gal}(\comp) \}$$ contains only finitely many isomorphisms classes  (see (\ref{Finitefamily})).

 We now turn our attention to the  holomorphic map
 $$\overline{V'_3} \supset V'_3 \to V' \to U \subset S.$$
As we are assuming that the projective surface $S$ is of general type, the result by Kobayashi mentioned above tells us that this map is, in fact, the restriction of a rational map  $\overline{V'_3} \dashrightarrow S$ between projective surfaces.

Finally, by a result of Maehara \cite[p. 102]{Maehara} the collection of birational classes of surfaces of general type that can arise as image of a fixed projective variety by a rational map, is finite. Hence, as $\{(\overline{V'_3})^{\sigma}\},$ $\sigma \in \mbox{Gal}(\comp),$ contains finitely many isomorphisms classes and $S$ is of general type, the collection $\{S^{\sigma}\}$ also contains finitely many birational classes and, as $S$ is   minimal, we conclude that $\{S^{\sigma}\}$ contains finitely many isomorphism classes; hence $S$ is arithmetic \cite[Criterion 2.1]{Gabino2}. This ends the proof of Theorem \ref{aritmetica}.

 \

To prove Theorem \ref{aritmeticaK}  we only need to show that if $S$ is arithmetic then the universal cover of the Kodaira family $S \to C$ is of arithmetic type. This can be accomplished as follows. By results of of Howard and Sommese (\cite[Theorem 2]{HowSom}; see also \cite{Koda1})
the number of surjective morphisms from $S$ to any Riemann surface of genus greater or equal to $2$ is finite. Therefore, if  $S$ is arithmetic,  the  criterion for arithmeticity we have been using throughout the paper implies that   the family  $S \to C$ is arithmetic; hence that its universal cover is of arithmetic type.

\

{\bf Acknowledgments.} The authors are grateful to their colleague Andrei Jaikin-Zapirain who generously told them how to prove Lemma \ref{andrei}.

\

\end{document}